\documentclass[11pt]{amsart}

\usepackage{amsfonts,amssymb,amsmath,amsthm,graphicx}
\usepackage{mathrsfs}

\newtheorem*{maintheorem}{Theorem}
\newtheorem{theorem}{Theorem}[section]
\newtheorem{proposition}[theorem]{Proposition}

\newtheorem{lemma}[theorem]{Lemma}
\newtheorem{corollary}[theorem]{Corollary}

\newtheorem{definition}[theorem]{Definition}

\numberwithin{equation}{section}

\newcommand{\la}{\lambda}
\newcommand{\N}{\mathbb{N}}
\newcommand{\Z}{\mathbb{Z}}

\newcommand{\R}{\mathbb{R}}
\newcommand{\C}{\mathbb{C}}

\newcommand{\diag}{\mathrm{diag}}
\newcommand{\tr}{\mathrm{tr}}
\newcommand{\ket}[1]{| #1 \rangle}
\newcommand{\bra}[1]{\langle #1 |}
\newcommand{\braket}[2]{\langle #1 | #2 \rangle}
\newcommand{\proj}[1]{\ket{#1}\bra{#1}}

\newcommand{\Sym}{\mathrm{Sym}}

\newcommand{\sgn}{\mathrm{sgn}}
\newcommand{\spa}[1]{\mathrm{span}\left\{{#1}\right\}}

\newcommand{\GLC}[1]{\ensuremath{\mathsf{GL}_{#1}(\C)}}

\newcommand{\GLR}[1]{\ensuremath{\mathsf{GL}_{#1}(\R)}}
\newcommand{\GL}{\mathsf{GL}}
\newcommand{\SL}{\mathsf{SL}}
\newcommand{\SG}{\mathsf{S}}
\newcommand{\B}{\mathsf{B}}

\newcommand{\Isot}[2]{{#1}({#2})}

\newcommand{\Specht}[1]{\mathscr{S}_{#1}}
\newcommand{\Weyl}[1]{\mathscr{V}_{#1}}

\newcommand{\kostka}[2]{\mathbf{K}_{{#1}{#2}}}

\newcommand{\per}{\mathrm{per}}
\renewcommand{\det}{\mathrm{det}}

\title{Even partitions in plethysms}

\author{Peter B\"urgisser}
\address{Peter B\"urgisser, Institute of Mathematics, University of Paderborn, D-33098 Paderborn, Germany}
\email{pbuerg@upb.de}

\author{Matthias Christandl}
\address{Matthias Christandl, Institute for Theoretical Physics, ETH Zurich, Wolfgang-Pauli-Strasse 27, CH-8093 Zurich, Switzerland}
\email{christandl@phys.ethz.ch}

\author{Christian Ikenmeyer}
\address{Christian Ikenmeyer, Institute of Mathematics, University of Paderborn, D-33098 Paderborn, Germany}
\email{ciken@math.upb.de}

\date{December 15, 2010}



\begin{document}
\sloppy

\maketitle

\begin{abstract}
We prove that for all natural numbers $k,n,d$ with $k \leq d$ and every partition $\lambda$
of size $kn$ with at most $k$ parts there exists an irreducible $\GLC d$-representation
of highest weight $2\lambda$ in the plethysm $\Sym^k(\Sym^{2n}\C^d)$.
This gives an affirmative answer to a conjecture by Weintraub (J. Algebra, 129 (1):~103-114, 1990).
Our investigation is motivated by questions of geometric complexity theory and uses ideas
from quantum information theory.
\end{abstract}

\setcounter{tocdepth}{1}

\section{Introduction}

\subsection{Statement of the result}
Geometric complexity theory is an approach to solve arithmetic versions of the famous
$\text{P}\neq \text{NP}$ conjecture and related questions
via geometric invariant theory~\cite{GCT1,GCT2}.
In this approach it is of interest to have conditions for the occurrence
of an irreducible $\GLC d$-representation~$\Weyl\mu$
with highest weight $\mu=(\mu_1,\ldots,\mu_d)$ in
the plethysm $\Sym^k(\Sym^{\ell}\C^d)$,
see~\cite[\S8.4]{BLMW:09} and Section~\ref{se:connGCT}.
Since $\Sym^k(\Sym^{\ell}\C^d)$ is homogeneous of degree~$k\ell$,
for the occurrence of $\Weyl \mu$,
it is necessary that $\mu$ has the size $|\mu| := \sum \mu_i = k\ell$.
Further, it is easy to see that $\mu$ should have at most~$k$ parts,
i.e., $k\le d$ and $\mu_i=0$ for $i>k$.

Weintraub~\cite{wei:90} conjectured that for the occurrence of~$\Weyl\mu$
in $\Sym^k(\Sym^{\ell}\C^d)$, where $\ell$ is even,
it is sufficient that all the parts~$\mu_i$
are even. This was supported by explicit computations for small
values of the parameters.
Manivel~\cite{mani:98} proved an asymptotic version of this conjecture
using geometric methods.

In this paper we prove Weintraub's conjecture.

\begin{maintheorem}\label{th:main}
For all $k,n,d\in \N$ with $k \leq d$ and for all partitions $\la$ of size~$kn$
with at most $k$ parts, the irreducible $\GLC d$-representation $\Weyl {2\la}$
of highest weight $2\la$ occurs in
the plethysm $\Sym^k(\Sym^{2n}\C^d)$­.
\end{maintheorem}

Here is a brief outline of the proof.
Let $V := \C^d$ and let $2\la(T)$~denote the isotypic component of $2\la$ in
$T:=(V^{\otimes 2n})^{\otimes k}$.
If no component of highest weight $2\la$ is contained in $\Sym^k(\Sym^{2n}V)$, then $\Sym^k(\Sym^{2n} V)$ and $2\la(T)$ are orthogonal. However, using techniques from quantum information theory,
we construct vectors $\ket v \in 2\la(T)$ and
$\ket \psi \in \Sym^k(\Sym^{2n}V)$
such that $\braket v \psi \neq 0$.
For this construction, which involves a sum of squares,
it is essential that all parts of the isotypic component are even.

\subsection{Connection to geometric complexity theory}\label{se:connGCT}

The most important open problem of algebraic complexity theory is
Valiant's Hypothesis~\cite{vali:79-3,vali:82},
which is an arithmetic analogue
of the famous P versus NP conjecture
(see \cite{bucs:96} for background information).
Valiant's Hypothesis can be easily stated in precise mathematical terms.

Consider the determinant $\det_d = \det [x_{ij}]_{1\le i,j\le d}$
of a $d$ by $d$ matrix of variables $x_{ij}$, and
for $m<d$, the {\em permanent} of its $m$ by $m$ submatrix defined as
$$
\per_m := \sum_{\sigma\in S_m} x_{1,\sigma(1)}\cdots x_{m,\sigma(m)} .
$$
We choose $z:=x_{dd}$ as a homogenizing variable and view
$\det_d$ and $z^{d-m}\per_m$
as homogeneous functions $\C^{d^2}\to \C$ of degree~$d$.
How large has $d$ to be in relation to~$m$ such that there is
a linear map $A\colon \C^{d^2}\to\C^{d^2}$ with the property that
\begin{equation}\tag{*}
z^{d-m}\per_m = \det_d \circ A ?
\end{equation}
It is known that such an $A$ exists for $d=O(m^2 2^m)$.
Valiant's Hypothesis states that (*)~is
impossible for $d$ polynomially bounded in~$m$.

Mulmuley and Sohoni~\cite{GCT1} suggested
to study an orbit closure problem related to~(*).
Note that the group $\GL_{d^2} = \GLC {d^2}$ acts on the space
$S^d (\C^{d\times d})^*$
of homogeneous polynomials of degree~$d$
in the variables $x_{ij}$ by substitution.
Instead of (*), we ask now whether
\begin{equation}\tag{**}
z^{d-m}\per_m\in \overline{\GL_{d^2}\cdot\det_{d}} .
\end{equation}
Mulmuley and Sohoni~\cite{GCT1} conjectured that
(**)~is impossible for $d=\text{poly}(m)$, which would imply
Valiant's Hypothesis.

Moreover, in~\cite{GCT1,GCT2} it was proposed to show
that (**) is impossible for specific values $m,d$ by
exhibiting an irreducible representation of $\SL_{d^2}$
in the coordinate ring of the orbit closure of $z^{d-m}\per_m$,
that does not occur in the coordinate ring
of $\overline{\GL_{d^2}\cdot\det_{d}}$.
We call such a representation of $\SL_{d^2}$
an {\em obstruction} for (**) for the values $m,d$.

We can label the irreducible $\SL_{d^2}$-representations
by partitions $\la$ into at most $d^2-1$ parts:
For $\la\in\N^{d^2}$ such that
$\la_1\ge\ldots\ge\la_{d^2-1} \ge \la_{d^2}=0$
we shall denote by $\Weyl\la (\SL_{d^2})$
the irreducible $\SL_{d^2}$-representation
obtained from the irreducible $\GL_{d^2}$-representation
$\Weyl \la$ with the highest weight~$\la$ by restriction.

If $\Weyl\la (\SL_{d^2})$ is an obstruction for $m,d$, then
we must have $|\la| = \sum_i \la_i = \ell d$ for some $\ell$,
see~\cite[Prop. 5.6.2]{BLMW:09}.
The irreducible representations $\Weyl\la (\SL_{d^2})$
not occurring in $\C[\overline{\GL_{d^2}\cdot\det_{d}}]$
are called candidates for obstructions and
can be characterized by the vanishing of certain
Kronecker coeffients, see~\cite[Prop. 5.2.1]{BLMW:09}
and \cite{BCI:09}.

For understanding which $\Weyl\la (\SL_{d^2})$ do occur in
$\C[\overline{\GL_{d^2}\cdot z^{d-m}\per_m}]$,
we first need to understand the case $m=d$.

\begin{proposition}\label{pro:perorbit}
Suppose that $|\la| = \ell d$. Then
$\Weyl\la (\SL_{d^2})$ occurs in
$\C[\overline{\GL_{d^2}\cdot \per_d}]$
iff there exist partitions $\mu,\nu$ of $\ell d$ into at most $d$ parts
with the following properties:
\begin{itemize}

\item[(1)] The tensor product $\Specht{\la} \otimes\Specht{\mu}\otimes \Specht{\nu}$ 
of irreducible representations of the symmetric group $\SG_{\ell d}$ 
contains the trivial representation 
(i.e., the corresponding Kronecker coefficient is strictly positive).
Moreover, if $\mu=\nu$, then $\Specht{\la} \otimes \mathrm{Sym}^2\Specht{\mu}$ 
contains the trivial representation.  

\item[(2)] $\Weyl \mu$ and $\Weyl \nu$ both occur in the plethysm $\Sym^d(\Sym^{\ell} \C^d)$.

\end{itemize}
\end{proposition}

\begin{proof}
This follows by combining
\cite[Prop. 4.4.1]{BLMW:09} with~\cite[Prop. 5.5.2]{BLMW:09}.
\end{proof}

It is thus essential for the geometric complexity program
to understand the plethysms appearing in
Weintraub's conjecture.
Our result shows that the second condition~(2) is satisfied in ``most cases''.

\subsection{Techniques from quantum information theory}\label{se:QIT}

It is one aim of quantum information theory to understand quantum correlations, or entanglement,
between spatially separated quantum systems and to utilise them in novel information processing applications.
A mathematical problem arising in this context is the quantum marginal problem. In its simplest form,
one wishes to determine the possible triple of spectra of the local quantum states
$\rho_A = \tr_{BC} \proj \psi, \rho_B=\tr_{AC} \proj \psi$ and $\rho_C=\tr_{AB} \proj \psi$,
where $\ket \psi \in \C^{d_A} \otimes \C^{d_B} \otimes \C^{d_C}$.
In~\cite{ChrMit06, Klyachko04, ChristHarrowMitch07} it was shown that this problem is related
to an asymptotic version of the problem of the positivity of the Kronecker coefficients of
the symmetric group (see condition (1) in Proposition~\ref{pro:perorbit}.).
In~\cite{ChrMit06, ChristHarrowMitch07} this relation is obtained by considering
the $n$-fold tensor product of $\ket \psi$ and projecting into the local isotypic components,
which are the analogues of typical subspaces in classical information theory.
For large~$n$, relations between the local spectra then transfer to relations between representations
of the symmetric group. For an application of this work in the context of geometric complexity theory
to understand which irreducible representations occur in
$\C[\overline{\GL_{d^2}\cdot\det_{d}}]$, see~\cite{BCI:09}.

The technique can also be applied in order to determine certain asymptotic behaviour of
the plethysm coefficient (see also~\cite{pauli}).
It is the challenge of this work to improve the asymptotical analysis
to obtain the precise statement of Weintraub's conjecture.

\section{Preliminaries}

\subsection{Highest weight vectors}

We briefly collect some facts about
the representation theory of $\GLC d$
and refer the reader to~\cite{FH:91} for details.
For our computational purposes it will be convenient
to use the bra and ket notation from physics.

A rational $\GLC d$-module~$M$ decomposes into its weight spaces
$$
M =  \bigoplus_{\la\in\Z^d} M_\la ,
$$
where $M_\la$ consists of the vectors $\ket v\in M$ such that
$$
\diag(t_1,\ldots,t_d)\cdot \ket v = t_1^{\la_1} \cdots t_d^{\la_d} \ket v
$$
for all $t_1,\ldots,t_d\in \C^\times$.
The vectors in $M_\la$ are said to be of {\em weight}~$\la=(\la_1,\ldots,\la_d)$.

Let $\B_d \subseteq \GLC d$ denote the group of upper triangular matrices in $\GLC d$.
A nonzero vector $\ket v\in M_\la$ is said to be a {\em highest weight vector} iff
the line $\C\, \ket {v}$ is stable under the action of $\B_d$.
The corresponding weight then satisfies $\la_1\ge\cdots\ge \la_d$.
It is well-known that the $\GLC d$-submodule generated by a highest weight vector
is irreducible and that a $\GLC d$-submodule is irreducible
iff it contains exactly one $\B_d$-stable line.
Hence one can assign to an irreducible $\GLC d$-module~$M$ the weight~$\la$
of their highest weight vectors
(which are uniquely determined up to a scalar).
One calls $\la$ the {\em highest weight of~$M$}.
Two irreducible $\GLC d$-modules are isomorphic iff they have the same
highest weight.

For each $\la\in\Z^d$ with nonincreasing components
there exists an irreducible $\GLC d$-module with highest weight~$\lambda$.
We denote it by $\Weyl \la$ and call it the {\em Weyl module of weight $\la$}.


\begin{lemma}\label{cor_rectweight}
Let $\la$ be a partition of $kn$ into at most $k$ parts, $k\le d$.
Then the $\GLC d$-module $\Weyl{\la}$ corresponding to $\la$
contains a nonzero vector of weight 
$(\underbrace{n,n,\ldots,n}_{k}, \underbrace{0, 0, \ldots, 0}_{d-k})$.
\end{lemma}

\begin{proof}
The dimension of the weight space of $\Weyl{\la}$ with the weight $\mu\in\N^{d}$
is given by the Kostka number $\kostka \la \mu$,
which is the number of semistandard Young tableaux
of shape $\la$ and content $\mu$, cf.~\cite[p.~368]{gw:09}.
It is known that $\kostka \la \mu > 0$ iff
$\lambda$ dominates $\mu$, that is,
$\sum_{j=1}^i \mu_j \leq \sum_{j=1}^i \la_j$ for all~$i$,
see~\cite[p.~26]{fult:97}.
On the other hand, it is clear that $\la$
dominates the rectangular partition $(n,n,\ldots,n, 0, 0, \ldots, 0)$.
\end{proof}

\subsection{Plethysms}

Plethysms are compositions of representations.
More precisely, given a representation $\rho\colon G\to \GLC \ell$
of a group~$G$ and a representation $\tau\colon \GLC \ell \to \GLC N$
of $\GLC \ell$,
the {\em plethysm} of $\rho$ and $\tau$ is defined as the composition
$\tau \circ \rho$.
In general, plethysms are reducible.
Their decomposition into irreducible representations
is a notoriously difficult problem
as the dimension of $\tau\circ\rho$ is typically large.

Here, we are concerned with $G=\GLC d$,
the representation~$\rho$ corresponds to $\Weyl {(2n, 0, 0, \ldots, 0)}$ ($d-1$ zeros)
and $\tau$~corresponds to $\Weyl {(k, 0, 0, \ldots, 0)}$ ($\ell-1$ zeros).
Since $\Weyl {(2n, 0, \ldots, 0)}$ can be realized as the symmetric power $\Sym^{2n} \C^d$,
i.e., the subspace of $(\C^d)^{\otimes 2n}$ of all vectors invariant under the action of $\SG_{2n}$,
we are concerned with the decomposition of the specific plethysm
$$
\Sym^k(\Sym^{2n}\C^d)
$$
into irreducibles.

\subsection{Isotypic components}

Let $M$ be a rational $\GLC d$-module.
The {\em isotypic component} $\la(M)$ of~$M$ of type $\la$
is defined as the sum of all irreducible submodules of $M$ with highest weight~$\la$.
Suppose that a unitary invariant Hermitian inner product~$\langle\ ,\ \rangle$
has been chosen on~$M$.
It is a well-known fact~\cite{FH:91}
that distinct isotypic components of $M$ are
orthogonal, so the isotypic decomposition
$M=\oplus_\la \la(M)$ is a decomposition into pairwise orthogonal subspaces.

In the following let $V:=\C^d$.
On the tensor space $T:=V^{\otimes q}$ we have two commuting linear actions,
of the general linear group~$\GLC d$
and the symmetric group $\SG_q$, namely,
for $g \in \GLC d$ and $\pi \in \SG_q$,
\begin{eqnarray*}
g (\ket{v_1} \otimes \cdots \otimes \ket{v_q}) &:=& g \ket{v_1} \otimes \cdots \otimes g  \ket{v_q} ,\\
\pi (\ket{v_1} \otimes \cdots \otimes \ket{v_q}) &:=& \ket{v_{\pi^{-1}(1)}} \otimes \cdots \otimes \ket{v_{\pi^{-1}(q)}}.
\end{eqnarray*}
The action of $S_q$ extends to the group algebra $\C[S_q]$.

We are interested in the isotypic components $\la(T)$ of $T$.
The highest weights~$\la$ occurring
consist of nonnegative numbers
and satisfy $|\la| = \sum_i \la_i =q$. Hence $\la$ is a partition of~$q$
into at most $d$ parts.
An element $a\in \C[\SG_q]$ acts as $\GLC d$-morphism and hence leaves $\la(T)$ invariant.

We proceed with a few simple observations.

\begin{lemma}\label{lem_tensorsquare}
If $\ket {v_\lambda}\in V^{\otimes p}$ is a highest weight vector of weight~$\la$
and $\ket {v_\mu}\in V^{\otimes q}$ is a highest weight vector of weight~$\mu$,
then $\ket {v_\lambda}\ket {v_\mu} := \ket {v_\lambda}\otimes \ket {v_\mu}$
is a highest weight vector in $V^{\otimes (p+q)}$ of
weight $\lambda +\mu$, where $(\lambda+\mu)_i=\la_i+\mu_i$.
\end{lemma}

\begin{proof}
Since $\C\,\ket {v_\la}$ and $\C\,\ket {v_\mu}$ are $\B_d$-stable lines,
it is clear that
$\ket {v_\lambda}\ket {v_\mu}$ is also $\B_d$-stable and of weight $\lambda+\mu$.
\end{proof}

\begin{definition}
Let $\ket {v_\lambda}\in V^{\otimes q}$ be a highest weight vector of weight~$\la$,
$g \in \GLC d$, and $\pi\in \SG_q$.
Then we call a vector of the form $\pi g \ket {v_\la}$
a {\em coherent state} of weight~$\la$.
\end{definition}

Lemma~\ref{lem_tensorsquare} immediately implies the following.

\begin{corollary} \label{cor_tensorsquare}
If $\ket w\in V^{\otimes q}$ is a coherent state of weight~$\la$, then
$\ket w \ket w$ lies in the isotypic component $(2\la)(V^{\otimes 2q})$.
\end{corollary}

Note that the standard Hermitian inner product on $V=\C^d$ extends
in a natural way to a unitary invariant inner product on
$T=V^{\otimes q}$.

Our proof strategy for the main theorem is based on
the following lemma. 

\begin{lemma}\label{lem_projection}
Let $V=\C^d$, $T=V^{\otimes k\ell}$,
and $S := \Sym^k(\Sym^{\ell}V) \subseteq T$.
Further, let $\mu$ be a highest weight for $\GLC d$.
If there exist vectors $\ket \psi \in S$ and $\ket v \in \mu(T)$
such that $\braket v \psi \neq 0$, then $\mu(S)$ is nonzero.
\end{lemma}

\begin{proof}
Suppose that $\mu(S)=0$.
We need to show that $S$ and $\mu(T)$ are orthogonal.
Since $S=\oplus_\la \la(S)$, it is sufficient to show that $\la(S)$
and $\mu(T)$ are orthogonal for all $\la\ne \mu$ (the case $\la=\mu$
being trivial because of $\mu(S)=0$).
But $\la(S)\subseteq \la(T)$ and $\la(T)$ and $\mu(T)$ are orthogonal
for $\la\ne \mu$.
\end{proof}

\subsection{Construction of highest weight vectors}

We will need an explicit construction for highest weight vectors in $V^{\otimes q}$.
Let $\{\ket{i} \}$ be the standard orthonormal basis for $\C^d$.
This defines an orthonormal basis of $V^{\otimes q}$ given by
$\{ \ket {i_1i_2\ldots i_q}\}$, $i_j \in \{1, \ldots, d\}$,
where we used the shorthand
$\ket {i_1i_2\ldots i_q}= \ket {i_1} \otimes \ket {i_1} \otimes \cdots \otimes \ket {i_q}$.
We call this basis the \emph{computational basis} borrowing terminology from quantum information theory.
For $1\le k\le d$ we define the following vector
(sometimes called Slater-determinant)
$$
\ket{v_k} := \sum_{\pi \in \SG_k} \sgn(\pi) \pi \ket{12\cdots k} \in V^{\otimes k} .
$$
It is easy to check that $\ket{v_k}$ is a highest weight vector of
weight $(1,\ldots,1,0,\ldots,0)$, where $1$ occurs $k$ times.

Let $\lambda$ be a partition of $q$ into at most $d$ parts.
The \emph{conjugate partition} $\la'$ corresponding to $\la$ is
defined by $\la'_j =|\{i\mid \la_i \ge j \}|$.
(The Young diagram corresponding to $\la'$ is obtained by reflecting
the one of~$\la$ at the main diagonal.)
It follows from Lemma~\ref{lem_tensorsquare} that
\begin{equation}\label{eq_explicit}
\ket {v_\la} := \ket{v_{\la'_1}} \ket{v_{\la'_2}} \cdots \ket{v_{\la'_{\ell(\la')}}}
\end{equation}
is a highest weight vector in $V^{\otimes q}$ of weight~$\la$, where $\ell(\la')$ denotes the number of nonzero parts of $\la'$.
(Note that each coefficient of~$\ket {v_\la}$,
when expanded in the computational basis, is a real number.)

The following is a direct corollary of Lemma~\ref{cor_rectweight}.

\begin{lemma}\label{crux}
Let $\lambda$ be a partition of $kn$ into at most~$k$ parts, $k\le d$.
Then the $\GLC d$-module $\Weyl \la$ generated by
$\ket {v_\la}$ 
contains a nonzero vector $\ket u$ of the form
\begin{equation}\label{eq:defu}
\ket u= \sum_{\pi \in \SG_{nk}}
\alpha_\pi \ \pi \cdot \ket{1^{\otimes n}2^{\otimes n} \cdots k^{\otimes n}}
\end{equation}
for some $\alpha_\pi \in \C$.
\end{lemma}

We also need the following well-known fact, which expresses
that $\R^n$ is dense in $\C^n$
for the Zariski topology.

\begin{lemma}\label{lem_permanence}
Let $f:\C^n \rightarrow \C$ be a polynomial function vanishing on $\R^n$.
Then $f$ vanishes on $\C^n$.
\end{lemma}

\section{Proof of the Theorem} 

Set $V := \C^d$, $T := V^{\otimes 2nk}$ and $S:=\Sym^k(\Sym^{2n}V) \subseteq T$.
Let $\lambda$ be a partition of $kn$ into at most~$k$ parts with $k\le d$.
By Lemma~\ref{lem_projection} it suffices to show
the existence of vectors $\ket \psi \in S$ and $\ket v \in \Isot {2\la} T$
such that $\braket v \psi \neq 0$.

We will now construct the vectors $\ket v$ and $\ket \psi$ before showing that their inner product is nonzero.
Let $\ket{v_\la}\in \Weyl{\la} \subseteq V^{\otimes nk}$
be the highest weight vector of weight $\la$
constructed as in (\ref{eq_explicit}).
By Lemma~\ref{crux} there exists a nonzero vector $\ket u \in \Weyl \la$
having the form~\eqref{eq:defu}.
Consider the function $f\colon\GLC d \rightarrow \C$ defined by
\begin{equation}\label{eq:fdef}
f(g) := \bra{v_\la} g \ket{u} = \sum_{\pi \in \SG_{nk}}\alpha_\pi
\bra{v_\la} g \pi \ket{1^{\otimes n}2^{\otimes n}\cdots k^{\otimes n}}.
\end{equation}
We claim that $f$ is nonzero. Indeed, since $\Weyl \la$ is irreducible, we have
$\spa{ \GLC{d} \cdot \ket {u} } = \Weyl \la $.
Hence there exist $g_i \in \GLC{d}$ and $\beta_i \in \C^\times$
such that
$$
\ket{v_\la} = \sum_i \beta_i g_i \ket{u}
$$
and therefore
$$
0\neq\braket {v_\la}{v_\la} =  \sum_{i} \beta_i \bra{v_\la} g_i \ket{u}
= \sum_i \beta_i f(g_i) .
$$
Hence $f(g_i)\ne 0$ for some~$i$.

The function~$f$ is a polynomial in the entries of $g$ and hence $f$~can be continued to
$f:\C^{d^2}\rightarrow \C$.
Lemma~\ref{lem_permanence} shows the existence of a {\em real} matrix
$g \in \R^{d\times d}$ with $f(g) \neq 0$.
By a slight perturbation of $g$ we may assume that $g\in\GLR d$.

As $f(g)\ne 0$ it follows from~\eqref{eq:fdef} that there exists $\pi \in \SG_{nk}$ such that
$$
\bra{v_\la} g \pi \ket{1^{\otimes n}2^{\otimes n}\cdots k^{\otimes n}} \neq 0.
$$
Hence the coherent state
$\ket w := \pi^{-1} g^{-1} \ket{v_\la}$
of weight~$\la$ satisfies
$$
\braket{w}{1^{\otimes n}2^{\otimes n} \cdots k^{\otimes n}} \neq 0.
$$
Moreover, $\ket w$ has only real coefficients in the computational basis,
because $\ket{v_\la}$ has only real coefficients, and $g$ is real.

Now we choose
$$
\ket \psi := \ket \phi ^{\otimes k} \in S \quad\text{ with }\quad
\ket \phi := \sum_{i=1}^d \ket i ^{\otimes 2n} \in \Sym^{2n}V.
$$
We have
$\ket \psi = \sum_{i_1=1}^d \sum_{i_2=1}^d\cdots\sum_{i_k=1}^d
\ket {i_1}^{\otimes 2n} \otimes \cdots \otimes \ket {i_k}^{\otimes 2n}$.
Let $\sigma \in \SG_{2nk}$ denote the permutation sorting
odd and even letters, more specifically,
$$
(\sigma(1),\ldots,\sigma(2nk))
= (1,3,5,\ldots, 2nk-1,2,4,6,\ldots,2nk) .
$$
For $\ket v := \sigma \ket w^{\otimes 2}$
we obtain
\begin{eqnarray}
\braket v \psi & = & \sum_{i_1=1}^d\sum_{i_2=1}^d\cdots\sum_{i_k=1}^d
\bra{w}^{\otimes 2} \sigma^{-1} \ket {i_1}^{\otimes 2n} \otimes \cdots \otimes \ket {i_k}^{\otimes 2n} \nonumber\\
& = & \sum_{i_1=1}^d\sum_{i_2=1}^d\cdots\sum_{i_k=1}^d
\Big( \braket{w}{{i_1}^{\otimes n} \cdots {i_k}^{\otimes n}} \Big)^2. \label{eqn_sumofsquares}
\end{eqnarray}
Since $\ket{w}$ has only real coefficients in the computational basis, each summand in (\ref{eqn_sumofsquares}) is a nonnegative real number.
Hence $\braket{w}{1^{\otimes n}2^{\otimes n} \cdots k^{\otimes n}} \neq 0$
ensures that $\braket v \psi > 0$.
Finally, since $\ket{w}$ is coherent, we have
$\ket w^{\otimes 2} \in \Isot{2\la}T$ by Corollary~\ref{cor_tensorsquare} and hence $\ket v \in \Isot{2\la}T$ by the invariance of $\Isot{2\la}T$ under the action of the permutation group.
\hfill$\Box$

\section*{Acknowledgements}
The authors acknowledge support of the German Science Foundation (DFG) Priority Programme 1388 on Representation Theory.
PB and CI are supported by DFG grant BU 1371/3-1. MC is supported by DFG grants CH 843/1-1 and CH 843/2-1, and by the Swiss National Science Foundation under grant PP00P2-128455.

\end{document}